\documentclass[a4paper]{amsart}

\usepackage{amsmath,amssymb,amsthm,amscd}
\usepackage[matrix,arrow,curve,tips]{xy}
            \SelectTips{cm}{}
\usepackage{mathrsfs}
\usepackage{stmaryrd}

\newtheorem{dfn}{Definition}[section]

\newtheorem{prop}[dfn]{Proposition}
\newtheorem{theo}[dfn]{Theorem}
\newtheorem{ex}[dfn]{Example}

\newcommand{\RR}{\mathbb{R}}

\newcommand{\cE}{\mathcal{E}}

\newcommand{\cS}{\mathcal{S}}

\newcommand{\fg}{\mathfrak{g}}

\newcommand{\fk}{\mathfrak{k}}
\newcommand{\fX}{\mathfrak{X}}

\newcommand{\GG}{\mathscr{G}}
\newcommand{\EE}{\mathscr{E}}
\newcommand{\GGG}{\mathscr{G}^{\scriptscriptstyle\#}}
\newcommand{\EEE}{E^{\scriptscriptstyle\#}}
\newcommand{\EEEp}{E'^{\scriptscriptstyle\#}}

\newcommand{\com}{\mathbin{{\scriptstyle\circ}}}
\newcommand{\ten}{\mathbin{\otimes}}

\newcommand{\germ}{\mathord{\mathrm{germ}}}

\newcommand{\src}{\mathord{\mathrm{s}}}
\newcommand{\trg}{\mathord{\mathrm{t}}}
\newcommand{\id}{\mathord{\mathrm{id}}}
\newcommand{\pr}{\mathord{\mathrm{pr}}}
\newcommand{\an}{\mathord{\mathrm{an}}}

\newcommand{\supp}{\mathord{\mathrm{supp}}}

\newcommand{\cu}{\mathord{\epsilon}}
\newcommand{\cm}{\mathord{\Delta}}
\newcommand{\C}{\mathord{\mathit{C}^{\infty}}}
\newcommand{\Cc}{\mathord{\mathit{C}^{\infty}_{c}}}
\newcommand{\G}{\mathord{\Gamma}}
\newcommand{\Gc}{\mathord{\Gamma_{c}}}
\newcommand{\Prim}{\mathord{{\mathrm{Prim}}}}
\newcommand{\Univ}{\mathord{{\mathcal{U}}}}
\newcommand{\Der}{\mathord{{\mathrm{Der}}}}

\newcommand{\End}{\mathord{{\mathrm{End}}}}
\newcommand{\gr}{\mathord{{\mathrm{gr}}}}
\newcommand{\N}{\mathord{{\mathrm{N}}}}
\newcommand{\Lt}{\mathord{{\mathrm{L}}}}
\newcommand{\Rt}{\mathord{{\mathrm{R}}}}
\newcommand{\Ad}{\mathord{{\mathrm{Ad}}}}
\newcommand{\Conj}{\mathord{{\mathrm{C}}}}
\newcommand{\T}{\mathord{{\mathrm{T}}}}

\newcommand{\Diff}{\mathord{\mathrm{Diff}}^{\Lt}}
\newcommand{\Op}{\mathord{\mathrm{\Omega}}}

\title[]
{Convolution bialgebra of a Lie groupoid and transversal distributions}

\author{J. Kali\v{s}nik}
\address{Faculty of mathematics and physics, University of Ljubljana,
         Jadranska 19, 1000 Ljubljana, Slovenia}
\address{Institute of Mathematics, Physics and Mechanics,
         University of Ljubljana, Jadranska 19,
         1000 Ljubljana, Slovenia}
\email{jure.kalisnik@fmf.uni-lj.si}

\author{J. Mr\v{c}un}
\address{Faculty of mathematics and physics, University of Ljubljana,
         Jadranska 19, 1000 Ljubljana, Slovenia}
\address{Institute of Mathematics, Physics and Mechanics,
         University of Ljubljana, Jadranska 19,
         1000 Ljubljana, Slovenia}
\email{janez.mrcun@fmf.uni-lj.si}

\thanks{The authors acknowledge the financial support
        from the Slovenian Research Agency
        (research core funding No. P1--0291).
        The first author also acknowledges the financial support
        from the Slovenian Research Agency
        grants No. J1--1690 and No. N1--0137.}


\begin{document}

\begin{abstract}
For a Lie groupoid $\GG$ over a smooth manifold $M$
we construct
the adjoint action of the \'etale Lie groupoid
$\GGG$ of germs of local bisections of $\GG$
on the Lie algebroid $\fg$ of $\GG$.
With this action, we form 
the associated convolution $\Cc(M)/\RR$-bialgebra
$\Cc(\GGG,\fg)$. We represent this $\Cc(M)/\RR$-bialgebra
in the algebra of transversal distributions on $\GG$.
This construction extends the Cartier-Gabriel
decomposition of the Hopf algebra
of distributions with finite support on a Lie group.
\end{abstract}

\maketitle

\section{Introduction}
Let $K$ be a Lie group and denote by $\cE'_{\text{fin}}(K)$ 
the Hopf algebra of distributions with finite support on $K$. 
The Lie algebra of primitive elements of $\cE'_{\text{fin}}(K)$ is 
isomorphic to the Lie algebra $\fk$ of $K$ and it generates a 
subalgebra of $\cE'_{\text{fin}}(K)$ which consists of all 
distributions supported at the unit of $K$. This 
subalgebra is in turn isomorphic to the universal enveloping algebra 
$\Univ(\fk)$ of $\fk$. On the other hand, the grouplike elements
of $\cE'_{\text{fin}}(K)$ correspond to Dirac measures on $K$
and generate a subalgebra of $\cE'_{\text{fin}}(K)$,
isomorphic to the group algebra $\RR K$ of the group $K$
equipped with the discrete topology.
The Hopf algebra
$\cE'_{\text{fin}}(K)$
has a natural Cartier-Gabriel decomposition as
the twisted tensor product
$K\ltimes \Univ(\fk)=\Univ(\fk)\ten\RR K$,
where $K$ acts on $\Univ(\fk)$ by the adjoint
representation~\cite{Car07}.

In this paper, we show how to extend this decomposition
to general Lie groupoids. For a Lie groupoid $\GG$ over a smooth
manifold $M$, we have the associated Lie algebroid $\fg$
with universal enveloping algebra
$\Univ(\fg)$,
which is a $\C(M)/\RR$-bialgebra~\cite{MoMrc10}.
Furthermore, we have the associated \'etale Lie groupoid
$\GGG$ of germs of local bisections of $\GG$ and the
Hopf algebroid $\Cc(\GGG)$ of functions with compact support
on $\GGG$~\cite{Mrc01}.
In Section~\ref{Convolution bialgebra of a Lie groupoid}
we show that we have an adjoint action of $\GGG$ on
the Lie algebroid $\fg$ as well as on the sheaf of germs of
the universal bialgebra $\Univ(\fg)$.
Using this adjoint action, we construct
the convolution $\Cc(M)/\RR$-bialgebra 
\[ \Cc(\GGG,\fg). \]
The elements of this bialgebra are suitable functions
with compact support on $\GGG$ with values in the sheaf of germs
of the universal algebra $\Univ(\fg)$.
Equivalently, we can describe $\Cc(\GGG,\fg)$
as the twisted tensor product
\[ \Univ(\fg)\ten_{\Cc(M)}\Cc(\GGG).  \]
In Section~\ref{Rep} we relate this bialgebra to
the algebra
$\cE'_{\trg}(\GG)$ of $\trg$-transversal
distributions on $\GG$ (see~\cite{ErYu19,LeMaVa17}).
Under the assumption that $\GG$ is Hausdorff and paracompact,
we construct a natural homomorphism of algebras
\[ \Phi_{\GG}:\Cc(\GGG,\fg)\to\cE'_{t}(\GG) \]
and explicitly compute its kernel. The distributions
in the image of this homomorphism have fiberwise
finite support.
If $\GG$ is a Lie group $K$, then the homomorphism
$\Phi_{K}$ is injective with 
$\text{Im}(\Phi_{K})=\cE'_{\text{fin}}(K)$ and gives
the isomorphism of the Cartier-Gabriel decomposition.
Furthermore, if $\EE$ is a Hausdorff paracompact
\'etale Lie groupoid,
then $\Phi_{\EE}$ is an isomorphism.

\section{Preliminaries}

Throughout this paper we will write
$M$ for a paracompact Hausdorff smooth manifold,
$\T(M)$ for its tangent bundle and
$\fX(M)$ for the Lie algebra of smooth vector fields on $M$.
We shall denote by $\C(M)$ the algebra of real smooth functions
on $M$ and by $\Cc(M)$ the algebra of real smooth
functions with compact support on $M$.

\subsection{Lie groupoids}
A {\em Lie groupoid\/} over $M$ is a
groupoid $\GG$ with objects $M$, equipped with a structure of a
smooth manifold (which may be non-paracompact non-Hausdorff)
such that all the groupoid structure maps of $\GG$ are smooth
and the source and the target
maps $\src,\trg:\GG\to M$ are submersions with paracompact
Hausdorff fibers.
For such a Lie groupoid $\GG$ and for any $x,y\in M$ we write
$\GG(x,-)=\src^{-1}(x)$, $\GG(-,y)=\trg^{-1}(y)$
and $\GG(x,y)=\GG(x,-)\cap\GG(-,y)$, and we
write $1_x\in\GG(x,x)$ for the unit arrow at $x$.
Furthermore, for any open subsets $U,V$ of $M$
we write $\GG(U,-)=\src^{-1}(U)$, $\GG(-,V)=\trg^{-1}(V)$
and $\GG(U,V)=\GG(U,-)\cap\GG(-,V)$. 
Note that $\GG(U,U)$ is an open Lie subgroupoid of $\GG$.
For any $g\in\GG(x,y)$ we have the left translation
$\Lt_{g}:\GG(-,x)\to\GG(-,y)$, $h\mapsto gh$,
and the right translation
$\Rt_g:\GG(y,-)\to\GG(x,-)$, $h\mapsto hg$.

Homomorphisms of Lie groupoids are smooth functors
between them.
A Lie groupoid
is {\em \'{e}tale\/} if all its structure maps 
are local diffeomorphisms.

For motivation, some historical remarks,
more details and many examples
of Lie groupoids, see~\cite{Mac05,MoMrc03,MoMrc05}.

\subsection{Local bisections of Lie groupoids}
Let $\GG$ be a Lie groupoid over $M$.
If $E$ is a subset of $\GG$ such
that $\src|_E$ is injective, we will
write $\alpha_E:\src(E)\to\GG$ for the assignment
determined by $\src\com\alpha_E=\id_{\src(E)}$ and
$\alpha_E(\src(E))=E$.
Similarly, if $E$ is a subset of $\GG$ such
that $\trg|_E$ is injective, we will
write $\beta_E:\trg(E)\to\GG$ for the assignment
determined by $\trg\com\beta_E=\id_{\trg(E)}$ and
$\beta_E(\trg(E))=E$.
A subset $E$ of $\GG$ is a {\em local bisection\/} of $\GG$
if both $\src|_E$ and $\trg|_E$ are injective,
the images
$\src(E)$ and $\trg(E)$ are open in $M$,
the maps $\alpha_E$ and $\beta_E$ are smooth and
the composition $\trg\com\alpha_E$ is a smooth open embedding
of smooth manifolds.
In particular, for any such local bisection $E$
we have the diffeomorphism 
$\tau_E:\src(E)\to\trg(E)$
satisfying $\alpha_E=\beta_E\com\tau_E$.
If $\alpha:U\to\GG$ is any smooth local section 
of the source map, defined on an open subset $U$ of $M$,
such that the composition $\trg\com\alpha$ is
a smooth open embedding of smooth manifolds,
then $\alpha(U)$ is a local bisection of $\GG$
with $\alpha_{\alpha(U)}=\alpha$.

The product of local bisections $E$ and $E'$ of $\GG$
is the local bisection
$E'\cdot E=\{g'g\,|\,g\in E,\,g'\in E',\, \src(g')=\trg(g)\}$
of $\GG$.
The inverse of a local bisection $E$ of $\GG$
is the local bisection $E^{-1}=\{g^{-1}\,|\,g\in E\}$ of $\GG$.

Local bisections $E$ and $E'$ of $\GG$ have the same
{\em germ\/} at an arrow $g\in E\cap E'$ if
there exists a local bisection $E''$ of $\GG$ such that
$g\in E''\subset E\cap E'$.
As usual, the {\em germ\/} of a local bisection $E$ of $\GG$ at $g\in E$,
denoted by $\germ_g(E)$, is the class of all local bisections
of $\GG$ through $g$ with the same germ at $g$
as the local bisection $E$.

The set of germs of all local bisections of $\GG$
has a natural structure of an \'{e}tale Lie groupoid
over $M$ (see~\cite{MoMrc03}), which we denote by
\[\GGG.\]
Indeed, for any local bisection $E$ of $\GG$
and $g\in E$ we define the source and the target
of $\germ_g(E)\in\GGG$ to be $\src(\germ_g(E))=\src(g)$
and $\trg(\germ_g(E))=\trg(g)$ respectively.
There is the obvious smooth structure
on $\GGG$ such that
the maps $\src,\trg:\GGG\to M$ are local diffeomorphisms.
The multiplication in $\GGG$ is induced by the multiplication
of local bisections. In particular, for any local bisections
$E$ and $E'$ of $\GG$ and for any $g\in E$ and $g'\in E'$ with
$\src(g')=\trg(g)$ we define
\[ (\germ_{g'}(E'))\cdot (\germ_g(E)) = \germ_{g'g}(E'\cdot E), \]
which yields
\[ (\germ_g(E))^{-1}= \germ_{g^{-1}}(E^{-1}).\]

For any local bisection $E$ of $\GG$ we have the associated
local bisection
\[ \EEE = \{ \germ_g(E) \,;\, g\in E\} \]
of $\GGG$, and the assignment $E\mapsto \EEE$ gives
a bijection between local bisections of $\GG$ and local
bisections of $\GGG$.

Note that we have a natural surjective
homomorphism of Lie groupoids over $M$
\[ \theta:\GGG\to\GG \]
which maps $\germ_g(E)$ to $g$.
This homomorphism
is an isomorphism if, and only if, the Lie groupoid
$\GG$ is \'etale.

If $\EE$ is an \'etale Lie groupoid, $E$ a local
bisection of $\EE$ and $e\in E$, then the germs
of the maps $\alpha_E$ and $\tau_E$ at $\src(e)$ depend only on $e$,
so we will write $\alpha_e=\germ_{\src(e)}\alpha_E$ and
$\tau_e=\germ_{\src(e)}\tau_E$. Similarly, we can also write
$\beta_e=\germ_{\trg(e)}\beta_E$.

\subsection{Lie algebroids}
A {\em Lie algebroid\/} over $M$
is a real smooth vector bundle $\fg\to M$ of finite rank,
equipped with a smooth map $\an:\fg\to \T(M)$
of vector bundles over $M$
and a Lie algebra structure on the
vector space $\G\fg$ of smooth global sections of $\fg$,
such that the induced map $\G(\an):\G\fg\to\fX(M)$
is a homomorphism of Lie algebras and
for any $X,Y\in\G\fg$ and any $f\in\C(M)$ we have
the {\em Leibniz identity\/}
\[ [X,fY]=f[X,Y]+\G(\an)(X)(f)Y. \]
For such a Lie algebroid $\fg$
and for any $X\in\G\fg$ and $f\in\C(M)$
we write $X(f)=\G(\an)(X)(f)$.
The map $\an$ is called the {\em anchor\/} of
the Lie algebroid $\fg$.
For more on Lie algebroids
see e.g.\ \cite{CdSW99,Mac05,MoMrc03,Pr67}.

For any Lie groupoid $\GG$ over $M$
we have an associated Lie algebroid over $M$,
which is defined as follows:
Write $\T^{\trg}(\GG)$ for the kernel of the derivative
of the target map of $\GG$. The pull-back
$\T_M^{\trg}(\GG)$ of the vector bundle $\T^{\trg}(\GG)$
along the unit map $M\to\GG$
has a natural structure of a Lie algebroid.
The anchor map of this Lie algebroid
is given by the restriction of the derivative of
the source map of $\GG$.
The sections of this Lie algebroid correspond
to the left invariant vector fields on $\GG$
(tangent to $\trg$-fibers),
and the Lie bracket of such sections is given by
the usual Lie bracket of the corresponding
left invariant vector fields.

Alternatively, we may consider
the kernel $\T^{\src}(\GG)$ of the derivative
of the source map and its pull-back
$\T_M^{\src}(\GG)$
along the unit map $M\to\GG$, which also
has a natural structure of a Lie algebroid,
with anchor given by the restriction of the
derivative of the target map.
The sections of this Lie algebroid correspond
to the right invariant vector fields on $\GG$
(tangent to $\src$-fibers),
and the Lie bracket of such sections is given by
the usual Lie bracket of the corresponding
right invariant vector fields.
Of course, the derivative of the inverse map of $\GG$
restricts to an isomorphism between
the Lie algebroids $\T_M^{\src}(\GG)$
and $\T_M^{\trg}(\GG)$.

In the literature, the model $\T_M^{\src}(\GG)$
is more common, except for the special case of
Lie groups, where it is standard to consider
the Lie algebra of left invariant vector fields.
It turns out that for the purpose of this paper
the model $\T_M^{\trg}(\GG)$ is more suitable, so
we will use this one.

\subsection{Lie-Rinehart algebras}
Let $\Bbbk$ be a field of characteristic $0$
and let $R$ be a commutative $\Bbbk$-algebra with local units.
Write $\Der_{\Bbbk}(R)$ for the Lie algebra of derivations on $R$.
A {\em $(\Bbbk,R)$-Lie algebra\/} is a Lie algebra
$L$ over $\Bbbk$ which is also a locally unital left $R$-module
and is equipped with a homomorphism of Lie algebras 
and left $R$-modules $\rho:L\to\Der_{\Bbbk}(R)$,
such that the Leibniz rule
\[ [X,rY]=r[X,Y]+\rho(X)(r)Y \]
holds for any $X,Y\in L$ and any $r\in R$.
For such a $(\Bbbk,R)$-Lie algebra $L$, the pair
$(R,L)$ is referred to as a {\em Lie-Rinehart algebra},
the homomorphism $\rho$ is called the {\em anchor},
and for any $X\in L$ and $r\in R$ we denote
$X(r)=\rho(X)(r)$.
For more on Lie-Rinehart algebras see
e.g.\ \cite{Her53,Hue04,Pal61,Rin63}.

If $\fg$ is a Lie algebroid over $M$,
then $\G\fg$ is an $(\RR,\C(M))$-Lie algebra.
In fact, the Serre-Swan theorem implies that
Lie algebroids over $M$ can be characterized
as the $(\RR,\C(M))$-Lie algebras which are
finitely generated and projective as $\C(M)$-modules.

\subsection{Bialgebras and Hopf algebroids}
Let $\Bbbk$ be a field of characteristic $0$
and let $R$ be a commutative $\Bbbk$-algebra with local units.
We say that a $\Bbbk$-algebra $A$ with local units
{\em extends\/} $R$ if $R$ is a subalgebra of $A$ and $A$
has local units in $R$.

Suppose that $A$ is a $\Bbbk$-algebra which extends $R$.
Then $A$ is an $R$-$R$-bimodule.
We shall write $A\ten_R A=A\ten_R^{ll} A$ for the tensor product
of left $R$-modules, which has two natural
right $R$-module structures.
We denote by $A\overline{\ten}_{R}A$ the left $R$-submodule
of $A\ten_R A$ consisting of those
elements of $A\ten_{R}A$ on
which both right $R$-actions coincide.
Note that $A\overline{\ten}_{R}A$ is a $\Bbbk$-algebra
with local units that extends $R$.

An {\em $R/\Bbbk$-bialgebra\/} is a $\Bbbk$-algebra $A$ 
which extends $R$,
equipped with a structure of a cocommutative
coalgebra in the category of left $R$-modules
(with comultiplication $\cm:A\to A\ten_R A$
and counit $\cu:A\to R$) such that
\begin{enumerate}
\item [(i)]   $\cm(A)\subset A\overline{\ten}_{R}A$,
\item [(ii)]  $\cu|_{R}=\id$,
\item [(iii)] $\cm|_{R}$ is the
              canonical embedding $R\subset A\ten_{R}A$,
\item [(iv)]  $\cu(ab)=\cu(a\cu(b))$ and
\item [(v)]   $\cm(ab)=\cm(a)\cm(b)$
\end{enumerate}
for any $a,b\in A$.

A {\em Hopf $R$-algebroid\/} is an $R/\Bbbk$-bialgebra $A$,
equipped with a $\Bbbk$-linear involution $S:A\to A$
(antipode) such that
\begin{enumerate}
\item [(vi)]   $S|_{R}=\id$,
\item [(vii)]  $S(ab)=S(b)S(a)$ for any $a,b\in A$ and
\item [(viii)] $\mu_{A}\com(S\ten\id)\com\cm=\cu\com S$,
\end{enumerate}
where $\mu_{A}$ denotes the multiplication in $A$.

Under different names and various generalities,
this type of structures have been studied
in~\cite{Nic85,NicWei82,Swe74,Tak77,Win74},
see also~\cite{Bo09,KalMrc13,
KowPos09,Lu96,Mal00,MoMrc10,Mrc01,Xu98}
and references therein.
Note that $\Bbbk/\Bbbk$-bialgebras are the usual
$\Bbbk$-bialgebras, often called just bialgebras.
In this paper, the name bialgebra will sometimes
be used
to refer to any general $R/\Bbbk$-bialgebra,
for simplicity of the exposition.

Let $A$ be an $R/\Bbbk$-bialgebra.
The {\em anchor\/} of $A$ is the homomorphism of $\Bbbk$-algebras
$\rho:A\to\text{End}_{\Bbbk}(R)$, defined by
$\rho(a)(r)=\cu(ar)$, for $a\in A$ and $r\in R$.
An element $a\in A$ is {\em primitive\/} if
\[ \cm(a)=\eta\ten a+a\ten\eta \]
for some $\eta\in R$ with $\eta a=a\eta=a$.
For any such primitive element $a$ we have
$\cu(a)=0$.
The set of primitive elements $\Prim(A)$ of $A$
is a left $R$-submodule of $A$ and has a natural
structure of a $(\Bbbk,R)$-Lie algebra~\cite{MoMrc10,Rin63},
with the restriction $\rho|_{\Prim(A)}$ as the anchor
and with the commutator as the Lie bracket.

\subsection{Universal enveloping algebra of a Lie-Rinehart algebra}
Let $\Bbbk$ be a field of characteristic $0$
and let $R$ be a commutative unital $\Bbbk$-algebra.
For any $(\Bbbk,R)$-Lie algebra $L$
we have the associated {\em universal enveloping algebra\/}
\[ \Univ(R,L), \]
which is a unital $\Bbbk$-algebra and is equipped with
a homomorphism of unital $\Bbbk$-algebras
$\iota_{R}:R\to\Univ(R,L)$
and a homomorphism of Lie algebras
$\iota_{L}:L\to\Univ(R,L)$ satisfying
$\iota_{R}(r)\iota_{L}(X)=\iota_{L}(rX)$
and $[\iota_{L}(X),\iota_{R}(r)]=\iota_{R}(X(r))$,
for any $r\in R$ and any $X\in L$
(see~\cite{NiWeXu99,Rin63}).

The algebra $\Univ(R,L)$
is determined by the
following universal property:
If $A$ is any unital $\Bbbk$-algebra,
$\kappa_{R}:R\to A$ any homomorphism of unital $\Bbbk$-algebras
and $\kappa_{L}:L\to A$ any homomorphism of Lie algebras
such that $\kappa_{R}(r)\kappa_{L}(X)=\kappa_{L}(rX)$
and $[\kappa_{L}(X),\kappa_{R}(r)]=\kappa_{R}(X(r))$
for any $r\in R$ and $X\in L$, then
there exists a unique homomorphism of
unital algebras $\kappa:\Univ(R,L)\to A$ such that
$\kappa\com \iota_{R}=\kappa_{R}$ and
$\kappa\com \iota_{L}=\kappa_{L}$.

By this universal property there is a unique
homomorphism of unital algebras
$\varrho:\Univ(R,L)\to \End_{\Bbbk}(R)$
such that $\varrho\com\iota_{L}=\rho$ and
$\varrho(\iota_{R}(r))(r')=rr'$ for all $r,r'\in R$.
The map $\iota_{R}$ is injective, hence we
identify $\iota_{R}(R)$ with $R$ and write simply
$\iota_{R}(r)=r$, for any $r\in R$.
In other words, the algebra $\Univ(R,L)$ extends $R$.
If $L$ is projective as a left $R$-module, then the map
$\iota_L$ is injective as well. For any $X\in L$ we
usually denote
$\iota_{L}(X)$ simply by $X$.
Using this notation, the algebra $\Univ(R,L)$
is generated by elements $X\in L$ and $r\in R$,
while the equalities $r\cdot X=rX$ and
$[X,r]=X\cdot r - r\cdot X=X(r)$ hold in $\Univ(R,L)$.

Furthermore, the universal enveloping algebra
$\Univ(R,L)$ is a filtered $\Bbbk$-algebra, with natural filtration
\[ \{0\}=\Univ^{(-1)}(R,L)\subset
   \Univ^{(0)}(R,L)\subset\Univ^{(1)}(R,L)
   \subset\Univ^{(2)}(R,L)\subset\cdots\;,\]
where $\Univ^{(0)}(R,L)=R$ and
$\Univ^{(n)}(R,L)$ is a vector subspace of $\Univ(R,L)$
spanned by $R$ and the
powers $\iota_L(L)^{m}$, $m=1,2,\ldots,n$.
The associated graded algebra
\[ \gr(\Univ(R,L))=\bigoplus_{n=0}^{\infty} \Univ^{(n)}(R,L)/\Univ^{(n-1)}(R,L) \]
is a commutative unital algebra over $R$.
If $L$ is projective as a left $R$-module,
this graded algebra is isomorphic to the symmetric algebra
of the $R$-module $L$,
by the Poincar\'{e}-Birkhoff-Witt theorem~\cite{Rin63}.

By the universal property of the algebra $\Univ(R,L)$,
there is a unique homomorphism of unital algebras
\[ \cm:\Univ(R,L)\to \Univ(R,L)\overline{\ten}_R\Univ(R,L) \]
such that $\cm(r)=1\ten r=r\ten 1$ and
$\cm(X)=1\ten X + X\ten 1$ for any $r\in R$ and $X\in L$.
This is a cocommutative coproduct on $\Univ(R,L)$
with counit
\[ \cu:\Univ(R,L)\to R \]
given by $\cm(u)=\varrho(u)(1)$ for any $u\in \Univ(R,L)$.
With this structure, the universal enveloping algebra
$\Univ(R,L)$ is an $R/\Bbbk$-bialgebra~\cite{MoMrc10}.

In particular, for any Lie algebroid
$\fg$ over $M$ we have the underlying
$(\RR,\C(M))$-Lie algebra $\G\fg$ with
the associated universal
enveloping $\C(M)/\RR$-bialgebra, which is denoted
simply by
\[ \Univ(\fg)=\Univ(\C(M),\Gamma\fg).\]

\subsection{Convolution bialgebra of an \'etale Lie groupoid}
Let $\EE$ be an \'etale Lie groupoid over $M$ and
let
\[ \Cc(\EE)=\Gc(\trg^{\ast}\mathcal{C}^\infty_M) \]
be the vector space of global
sections with compact support of the
pull-back of the $\EE$-sheaf $\mathcal{C}^\infty_M$ 
of germs of smooth real functions on $M$
along the target map
$\trg:\EE\to M$ (see~\cite{CraMo00}).
There is an associative convolution product
on $\Cc(\EE)$ given by
\[
(b'\cdot b)(e'')=\sum_{e''=e'e}b'(e')(e'\cdot b(e))
\]
for any $b,b'\in\Cc(\EE)$. Note that this sum is in fact finite
because $b$ and $b'$ are sections with compact support.
Furthermore, we have $b(e)\in (\mathcal{C}^\infty_M)_{\trg(e)}$,
so $e'\cdot b(e)\in (\mathcal{C}^\infty_M)_{\trg(e')}$
and the product $b'(e')(e'\cdot b(e))$ is defined in
the stalk $(\mathcal{C}^\infty_M)_{\trg(e'')}$.

An example of an element of $\Cc(\EE)$ is given by
a local bisection $E$ of $\EE$ and a smooth real function
$f\in \Cc(\trg(E))$: the associated
section $\langle f,E\rangle \in \Cc(\EE)$ is defined by
\[
\langle f,E\rangle (e)= \germ_{\trg(e)}f
\in (\mathcal{C}^\infty_M)_{\trg(e)}
\]
for all $e\in E$ end equals $0$ outside $E$.
By definition, elements of
$\Cc(\EE)$ are finite sums of sections of this form.
If $E'$ is another local bisection of $\GG$ and
$f'\in \Cc(\trg(E'))$, then the convolution product
is given by
\[ \langle f',E'\rangle\cdot\langle f,E\rangle
= \langle f' \cdot (f\com\tau^{-1}_{E'}) ,E'\cdot E\rangle .\]

Since we identify $M$ with the image of the
unit map $M\to\EE$, the algebra $\Cc(\EE)$ extends
the algebra $\Cc(M)$. In particular, the algebra $\Cc(\EE)$
is a $\Cc(M)$-$\Cc(M)$-bimodule.
The diagonal map $\EE\to\EE\times_M\EE$ induces
a coproduct map~\cite{Mrc07_2}
\[ \cm:\Cc(\EE) \to \Cc(\EE)\ten_{\Cc(M)}\Cc(\EE), \]
the target map $\EE\to M$ induces a counit map
\[ \cu:\Cc(\EE) \to \Cc(M) \]
and the inverse map $\EE\to\EE$ induces an antipode map
\[ S:\Cc(\EE) \to \Cc(\EE). \]
Explicitly, on the element
$\langle f,E\rangle$ as above
we have
\[ S(\langle f,E\rangle) = \langle f\com\tau_E,E^{-1}\rangle, \]
\[ \cu(\langle f,E\rangle)=f \]
and
\[ \cm(\langle f,E\rangle) =
   \langle f,E\rangle \ten \langle \eta ,E\rangle,\]
where $\eta$ is any function in $\Cc(\trg(E))$
which equals $1$
on a neighbourhood of the support of $f$.

With this structure, the algebra $\Cc(\EE)$
is a Hopf $\Cc(M)$-algebroid~\cite{Mrc01,Mrc07_1}.

\section{The convolution bialgebra of a Lie groupoid}\label{Convolution bialgebra of a Lie groupoid}

The aim of this section is to construct the convolution
$\Cc(M)/\RR$-bialgebra $\Cc(\GGG,\fg)$,
for any Lie groupoid $\GG$ over $M$
with its Lie algebroid $\fg$. We will in fact first construct
a $\Cc(M)/\RR$-bialgebra $\Cc(\EE,\fg)$ for any action
of an \'etale
Lie groupoid $\EE$ over $M$
on an arbitrary Lie algebroid $\fg\to M$.
The $\Cc(M)/\RR$-bialgebra $\Cc(\GGG,\fg)$ is then a special case
of this construction when applied to the adjoint $\GGG$-action,
which we also introduce.

\subsection{Actions of \'etale Lie groupoids on Lie algebroids}\label{Actions of etale Lie groupoids on Lie algebroids}
Let $\EE$ be an \'etale Lie groupoid over $M$
and let $\pi:\fg\to M$ be a Lie algebroid over $M$.
Note that for any open subset $U$ of $M$, the restriction
$\fg|_U=\pi^{-1}(U)$ is a Lie algebroid over $U$.

Let $\EE\times_M \fg \to\fg$ be
a left $\EE$-action on the vector
bundle $\pi:\fg\to M$.
For any local bisection $E$ of $\EE$
we have an isomorphism of vector bundles
$\tau_E^{\fg}:\fg|_{\src(E)}\to\fg|_{\trg(E)}$ over
$\tau_E: \src(E)\to\trg(E)$ given by
\[ \tau_E^{\fg}(\xi)= \alpha_{E}(\pi(\xi))\cdot \xi, \]
for any $\xi\in\fg|_{\src(E)}$.
If this isomorphism $\tau_E^{\fg}$ is in fact an isomorphism
of Lie algebroids over the diffeomorphism $\tau_E$,
for any local bisection $E$ of $\EE$,
then we say that the $\EE$-action respects the Lie algebroid
structure, or simply, that this $\EE$-action is an action
of the \'etale Lie groupoid $\EE$ on the
Lie algebroid $\fg$.

Assume that we have such an action
of the \'etale Lie groupoid $\EE$ on the
Lie algebroid $\fg$. For any local bisection $E$ of $\EE$,
the isomorphism $\tau_E^{\fg}$ induces an isomorphism
of Lie-Rinehart algebras
\[ \G(\tau_E^\fg): (\C(\src(E)),\G(\fg|_{\src(E)}))
   \to (\C(\trg(E)),\G(\fg|_{\trg(E)})) \]
and an isomorphism of universal enveloping bialgebras
\[ \Univ(\tau_E^{\fg}): \Univ(\fg|_{\src(E)})
   \to \Univ(\fg|_{\trg(E)}) \]
over the isomorphism of algebras
$\C(\tau_E^{-1}):\C(\src(E))\to\C(\trg(E))$.

Let $x$ be a point in $M$. The smooth real functions on $M$
with zero germ at $x$ form an ideal in
$\C(M)$, which we shall denote by $\N_x(\C(M))$. For any
left $\C(M)$-module $\mathcal{M}$ we consider the submodule
\begin{align*}
\N_x(\mathcal{M})
&= \N_x(\C(M))\mathcal{M} 
 = \{ fm \,;\, f\in\N_x(\C(M)),\, m\in\mathcal{M} \} \\
&= \{ m\in\mathcal{M}\,;\, fm=m \mbox{ for some }
      f\in\N_x(\C(M)) \} \\
&= \{ m\in\mathcal{M}\,;\, fm=0 \mbox{ for some }
      f\in\C(M)\mbox{ with } \germ_x(f)=1 \} 
\end{align*}
and write
\[ \mathcal{M}_x = \mathcal{M}/\N_x(\mathcal{M})\]
for the associated localization.
We will denote by $m_x$ the localization 
(or the germ) of $m\in \mathcal{M}$
at $x$, which is the image of $m$ along the quotient
map $\mathcal{M}\to\mathcal{M}_x$.
In particular, the localization
$\C(M)_x = \C(M)/\N_x(\C(M))$
is the algebra of germs of smooth real functions on $M$ at $x$,
and we have a natural isomorphism of left $\C(M)_x$-modules
$\mathcal{M}_x \cong \C(M)_x\ten_{\C(M)} \mathcal{M}$.
In the same way we can define a localization
of left $\Cc(M)$-modules and we can observe that
$\C(M)_x=\Cc(M)_x$.

We define the germ of the Lie algebroid $\fg$ at $x$
as the localization
\[ (\G\fg)_x. \]
The Lie algebroid structure of $\fg$ induces
an $(\RR,\C(M)_x)$-Lie algebra structure on the quotient
$(\G\fg)_x$: 

\begin{prop}
The germ $(\G\fg)_x$ of the Lie algebroid $\fg$ at $x$
is an $(\RR,\C(M)_x)$-Lie algebra.
Elements of $(\G\fg)_x$ are germs $X_x$ of smooth
sections $X$ of $\fg$ at $x$.
For any open neighbourhood $U$ of $x$ in $M$
we have $(\G\fg)_x=(\G\fg|_{U})_x$.
\end{prop}

\begin{proof}
We need to check that $\an(\N_x(\G\fg))\subset \N_x(\fX(M))$
and that $\N_x(\G\fg)$ is an ideal in
the Lie algebra $\G\fg$.

Take any $Y\in\N_x(\G\fg)$ and choose
a function $f\in\C(M)$ with $fY=Y$ and
$f_x=\germ_{x}(f)=0$.
Then we have
\[ \an(Y)=\an(fY)=f\an(Y)\in \N_x(\fX(M)). \]
Furthermore, for any
$X\in\G\fg$ we have
$X(f)_x=0$ and
\[ [X,Y]=[X,fY]=f[X,Y]+X(f)Y \in \N_x(\G\fg)\,.\qedhere \]
\end{proof}

For any local bisection $E$ of $\EE$ and
any $e\in E$, the isomorphism
$\G(\tau_E^\fg)$ induces an isomorphism 
of Lie-Rinehart algebras over $\C(\tau_e^{-1})$
\[ \G(\tau_E^\fg)_{\src(e)}:
   (\C(M)_{\src(e)},\G(\fg)_{\src(e)})
   \to
   (\C(M)_{\trg(e)},\G(\fg)_{\trg(e)}).
\]
Since the Lie groupoid $\EE$ is \'etale,
this isomorphism depends only on $e$ and not on
the choice of the local bisection $E$. We shall therefore
denote this isomorphism simply by
\[ (\G\fg)_{\src(e)} \to (\G\fg)_{\trg(e)},
\qquad X_{\src(e)}\mapsto e\cdot X_{\src(e)},\]
as an action of $\EE$ on the sheaf of germs of
sections of $\fg$.
We can summarize 
that the germs of sections of $\fg$ form
an $\EE$-sheaf of Lie-Rinehart algebras.

Similarly, we take the localization
\[ \Univ(\fg)_x \]
of the universal enveloping bialgebra.
The bialgebra structure of $\Univ(\fg)$
induces a $\C(M)_x/\RR$-bialgebra structure
on $\Univ(\fg)_x$: 

\begin{prop}
For any $x\in M$,
the localization $\Univ(\fg)_x$
of the universal enveloping bialgebra
$\Univ(\fg)$
is a $\C(M)_x/\RR$-bialgebra, and we have
a natural isomorphism
\[ \Univ(\fg)_x \to \Univ(\C(M)_x,(\G\fg)_x) \]
of $\C(M)_x/\RR$-bialgebras.
For any open neighbourhood $U$ of $x$ in $M$
we have $\Univ(\fg)_x=\Univ(\fg|_U)_x$. 
\end{prop}

\begin{proof}
Let us abbreviate
$\Univ=\Univ(\fg)$.

(i) First we show by induction on $n\in\{0,1,2,\ldots\}$ that
\[ \Univ^{(n)} \cdot \N_x(\Univ) \subset \N_x(\Univ). \]

Take any $u\in\N_x(\Univ)$ and choose
$f\in\C(M)$ such that $fu=u$ and $f_x=0$.
For
any $f'\in\C(M)=\Univ^{(0)}$ we have
$f'u=f'(fu)=f(f'u)\in\N_x(\Univ)$,
so
\[ \Univ^{(0)} \cdot \N_x(\Univ) \subset \N_x(\Univ). \]
For any $X\in\G\fg$ we have
$X(f)_x=0$ and
$Xu = X (fu)= (fX+X(f))u\in\N_x(\Univ)$,
which shows that
\[ (\G\fg) \cdot \N_x(\Univ)\subset \N_x(\Univ).\]
Now assume that 
$\Univ^{(n)} \cdot \N_x(\Univ) \subset \N_x(\Univ)$
for some $n\in\{0,1,2,\ldots\}$. Then we have
\[ (\G\fg)\cdot \Univ^{(n)} \cdot \N_x(\Univ)
\subset (\G\fg)\cdot \N_x(\Univ) \subset \N_x(\Univ),
\]
and this yields
\[ \Univ^{(n+1)} \cdot \N_x(\Univ) \subset \N_x(\Univ).\]

(ii) Part (i) implies that
\[ \Univ \cdot \N_x(\Univ) \subset \N_x(\Univ), \]
so $\N_x(\Univ)$ is a left ideal in $\Univ$.
We can easily see that
$\N_x(\Univ)$ is in fact a two-sided ideal
in $\Univ$ and that
$\Univ_x$
is a $\C(M)_x/\RR$-bialgebra.

(iii)
Let us further abbreviate $L=\G\fg$ and $R=\C(M)$.
The natural embeddings $\iota_R:R\to\Univ$ and
$\iota_L:L\to\Univ$ induce linear maps
\[ \kappa_{R_x}:R_x\to \Univ_x,\qquad f_x
   =f+\N_x(R)\mapsto f+\N_x(\Univ), \]
and
\[ \kappa_{L_x}:L_x\to \Univ_x,\qquad X_x
   =X+\N_x(L)\mapsto X+\N_x(\Univ). \]
Using the fact that $\N_x(\Univ)$ is an ideal in $\Univ$,
it is straightforward to check that 
$\kappa_{R_x}$ is a homomorphism of algebras,
$\kappa_{L_x}$ is a homomorphism of Lie algebras,
$\kappa_{R_x}(f_x)\kappa_{L_x}(X_x)=\kappa_{L_x}(f_x X_x)$ and
$[\kappa_{L_x}(X_x),\kappa_{R_x}(f_x)]=\kappa_{R_x}(X_x (f_x))$.
By the universal property, the maps $\kappa_{R_x}$ and
$\kappa_{L_x}$ determine a homomorphism of algebras
\[ \kappa: \Univ(R_x,L_x) \to \Univ_x. \]

(iv)
We use the universal property to define a homomorphism
of algebras
\[ \psi: \Univ \to \Univ(R_x,L_x), \]
determined by the homomorphism of algebras
\[ \psi_R:R\to \Univ(R_x,L_x),\qquad f\mapsto f_x\,,\]
and the homomorphism of Lie algebras
\[ \psi_L:L\to \Univ(R_x,L_x),\qquad X\mapsto X_x\,.\]
Again, it is straightforward to check
$\psi_R(f)\psi_L(X)=\psi_L(fX)$ and
$[\psi_L(X),\psi_R(f)]=\psi_R(X(f))$, so the
algebra homomorphism $\psi$ is well defined.
This homomorphism induces an algebra homomorphism
\[ \Univ_x \to \Univ(R_x,L_x), \]
which is the inverse to the homomorphism $\kappa$
defined in part (iii). Indeed, we have
\[
\psi(\N_x(\Univ))
=  \psi(\N_x(R)\cdot\Univ)
\subset \psi(\N_x(R))\cdot\psi(\Univ)
\subset \{0\}\cdot \psi(\Univ)=\{0\}\,.\qedhere
\]
\end{proof}

For any local bisection $E$ of $\EE$ and
any $e\in \EE$, the isomorphism
$\Univ(\tau_E^\fg)$ induces an isomorphism 
of bialgebras over $\C(\tau_e^{-1})$
\[ \Univ(\tau_E^\fg)_{\src(e)}:
   \Univ(\fg|_{\src(E)})_{\src(e)}
   \to
   \Univ(\fg|_{\trg(E)})_{\trg(e)}.
\]
Since the Lie groupoid $\EE$ is \'etale,
this isomorphism
depends only on $e$ and not on
the choice of the local bisection $E$,
so we will write
\[ \Univ(\tau_E^\fg)_{\src(e)}(u_{\src(e)})
   = e\cdot u_{\src(e)}, \]
for any
$u\in\Univ(\fg|_{\src(E)})$.
By the previous proposition, this isomorphism
gives us
an isomorphism of universal enveloping bialgebras
\[ \Univ(\C(M)_{\src(e)},(\G\fg)_{\src(e)})
   \to
   \Univ(\C(M)_{\trg(e)},(\G\fg)_{\trg(e)}),
\]
which maps
$X_{\src(e)}\mapsto e\cdot X_{\src(e)}$
and $f_{\src(e)}\mapsto f_{\src(e)}\com\tau_e^{-1}$,
for any $X\in\G\fg$ and any $f\in\C(M)$.

We may naturally equip the space
\[ \cS(\fg) = \coprod_{x\in M} \Univ(\fg)_x \]
with a structure of a sheaf over $M$, and we just
demonstrated that $\cS(\fg)$ is actually an $\EE$-sheaf
of bialgebras. This sheaf is filtered by $\EE$-subsheaves
of modules
\[ 
\cS^{(0)}(\fg)\subset
\cS^{(1)}(\fg)\subset
\cS^{(2)}(\fg)\subset \cdots,\]
where
\[ \cS^{(n)}(\fg)
    = \coprod_{x\in M} \Univ^{(n)}(\fg)_x.\]
Note that $\cS^{(0)}(\fg)=\mathcal{C}^\infty_M$.
By the Poincar\'{e}-Birkhoff-Witt theorem it follows that 
\[ 
\Univ^{(n)}(\fg)=\G\cS^{(n)}(\fg),
\]
while
\[ 
\Univ(\fg)=\lim_{\to n}\G\cS^{(n)}(\fg).
\]

\subsection{The bialgebra of an action of $\EE$ on $\fg$}\label{Bialgebra of an action}
Let $\EE$ be an \'etale Lie groupoid over $M$ which acts
on a Lie algebroid $\fg$ over $M$, as in the previous subsection.
For any $n\in\{0,1,2,\ldots\}$ denote
\[ \Cc(\EE,\fg)^{(n)}=\Gc(\trg^{\ast}\cS^{(n)}(\fg)) \]
and write
\[ \Cc(\EE,\fg) = \lim_{\to n} \Cc(\EE,\fg)^{(n)}. \]
We define a convolution product on the
$\Cc(M)$-module $\Cc(\EE,\fg)$ as follows:
for any $a,a'\in\Cc(\EE,\fg)$ and any $e''\in \EE$ we define
\[ (a'\cdot a)(e'')
   = \sum_{e''=e'e} a'(e') (e'\cdot a(e)).
\]
This sum is finite
because $a$ and $a'$ are sections with compact support,
we have $a(e)\in \cS(\fg)_{\trg(e)}$,
so $e'\cdot a(e)\in \cS(\fg)_{\trg(e')}$
and the product $a'(e')(e'\cdot a(e))$ is defined in
the stalk $\cS(\fg)_{\trg(e'')}$.
With this product, the $\Cc(M)$-module $\Cc(\EE,\fg)$
becomes an algebra with local units 
that extends $\Cc(M)$. Furthermore, we have
$\Cc(M)\subset \Cc(\EE)\subset\Cc(\EE,\fg)$.

An example of an element of $\Cc(\EE,\fg)$ is given by
a local bisection $E$ of $\EE$ and an element
$u\in\Univ(\fg)$ with compact support $\supp(u)$
in $\trg(E)$:
the associated
section $\langle u,E\rangle \in \Cc(\EE,\fg)$ is defined by
\[
\langle u,E\rangle (e)= u_{\trg(e)}
\]
for all $e\in E$ end equals $0$ outside $E$.
By definition, elements of
$\Cc(\EE,\fg)$ are finite sums of sections of this form.
If $E'$ is another local bisection of $\GG$ and
$u'\in\Univ(\fg)$ has compact support in $\trg(E')$,
then the convolution product
is given by
\[ \langle u',E'\rangle\cdot\langle u,E\rangle
= \langle u' \cdot (\Univ(\tau^\fg_{E'})(u)) ,E'\cdot E\rangle .\]
Note that this expression makes sense.
Indeed, while the element $\Univ(\tau^\fg_{E'})(u)$ is
not globally defined,
the product $u' \cdot (\Univ(\tau^\fg_{E'})(u))$
has compact support in $\trg(E'\cdot E)$ and can be extended by zero
to a global element of $\Univ(\fg)$. 

The algebra $\Cc(\EE,\fg)$ is generated by elements of the form
$\langle f,E\rangle$ and $\langle X,E\rangle$,
where $E$ is a local bisection of $\EE$ and where $f\in\C(M)$ and
$X\in\G\fg$ both have compact support in $\trg(E)$.
If $E'$ is another local bisection of $\EE$ and
$f'\in\C(M)$ and
$X'\in\G\fg$ both have compact support in $\trg(E')$, we have
\begin{align*}
\langle f',E'\rangle \cdot \langle f,E\rangle
&= \langle f'\cdot (f\com\tau_{E'}^{-1}),E'\cdot E \rangle,\\
\langle X',E'\rangle \cdot \langle f,E\rangle
&= \langle X'\cdot (f\com\tau_{E'}^{-1}),E'\cdot E \rangle,\\
\langle f',E'\rangle \cdot \langle X,E\rangle
&= \langle f'\cdot \G(\tau_{E'}^\fg)(X),E'\cdot E \rangle,\\
\langle X',E'\rangle \cdot \langle X,E\rangle
&= \langle X'\cdot \G(\tau_{E'}^\fg)(X),E'\cdot E \rangle.
\end{align*}

The vector space $\Cc(\EE,\fg)$ can be represented
as the tensor product
\[ \Cc(\EE,\fg) \cong \Univ(\fg) \ten_{\Cc(M)} \Cc(\GGG), \]
see~\cite{GrHaVa78}.
By this isomorphism, the element
$\langle u,E\rangle \in \Cc(\EE,\fg)$ as above corresponds to the
tensor
$u\ten \langle \eta,E\rangle$,
where $\eta\in \Cc(\trg(E))$ is any function which equals $1$ on
a neighbourhood of the support of $u$. 

We may check that $\Cc(\EE,\fg)$
is a $\Cc(M)/\RR$-bialgebra. 
The counit and the comultiplication
on $\Cc(\EE,\fg)$ are given by
\[ \cu(\langle u,E\rangle) = \cu(u) \]
and
\[ \cm(\langle u,E\rangle) = \sum_i
\langle u^{(1)}_i,E\rangle \ten \langle u^{(2)}_i,E\rangle, \]
where $\cm(u)= \sum_i u^{(1)}_i \ten u^{(2)}_i$
is the coproduct of $u$ in $\Univ(\fg)$ and
$u^{(1)}_i, u^{(2)}_i\in\Univ(\fg)$ all have
compact support in $\trg(E)$.

The elements of the form
$\langle X,M\rangle\in\Cc(\EE,\fg)$ as above
are primitive and generate, together with
$\Cc(M)\subset\Cc(\EE,\fg)$, the image
of the algebra embedding
$\Univ_{c,M}(\fg)=\Cc(M)\cdot\Univ(\fg)\to\Cc(\EE,\fg)$,
$u\mapsto \langle u,M\rangle$.
With these primitive elements one can reconstruct
the Lie algebroid $\fg$~\cite{MoMrc10}.
On the other hand, the elements of the form
$\langle f,E\rangle\in\Cc(\EE,\fg)$ as above
are weakly grouplike~\cite{Mrc07_2} and generate the subalgebra
$\Cc(\EE)\subset\Cc(\EE,\fg)$.
These weakly grouplike elements can be used to
reconstruct the sheaf $\trg:\EE\to M$~\cite{Mrc07_2}.
To reconstruct the Lie groupoid structure on $\EE$
one would in general need the presence
of the antipode~\cite{Mrc07_1}.

A special case of this construction, for an action
of an \'etale Lie groupoid on
a bundle of Lie algebras,
was given in~\cite{KalMrc13}.
In this case, the resulting bialgebra is in fact
a Hopf $\Cc(M)$-algebroid.

\subsection{The adjoint action and the convolution bialgebra}\label{The adjoint action}
Let $\GG$ be a Lie groupoid and let 
$\fg$ be the Lie algebroid 
$\pi:\T_M^{\trg}(\GG)\to M$
associated to $\GG$.
We will now show that the \'etale Lie groupoid
$\GGG$ of germs of local bisections of $\GG$ acts naturally
on the Lie algebroid $\fg$ by the adjoint representation.
Consequently, we obtain the convolution
$\Cc(M)/\RR$-bialgebra
\[ \Cc(\GGG,\fg) \]
of the Lie groupoid $\GG$.

To construct the adjoint action of $\GGG$ on $\fg$, note
that any local bisection $E$ of $\GG$ induces a diffeomorphism
$\Lt_E: \GG(-,\src(E))\to \GG(-,\trg(E))$ given by left translation,
\[ \Lt_E(h)=\alpha_E(\trg(h))\cdot h, \qquad h\in\GG(-,\src(E)), \]
as well as a diffeomorphism
$\Rt_E: \GG(\trg(E),-)\to \GG(\src(E),-)$ given by right translation,
\[ \Rt_E(h)=h\cdot \beta_E(\src(h)), \qquad h\in\GG(\trg(E),-). \]
The conjugation diffeomorphism
$\Conj_E:\GG(\src(E),\src(E))\to\GG(\trg(E),\trg(E))$
is given by
\begin{align*}
\Conj_E(h)
&= \Rt_{E^{-1}}(\Lt_E(h)) = \Lt_E(\Rt_{E^{-1}}(h))
 = \Rt_{E}^{-1}(\Lt_E(h)) = \Lt_E(\Rt_{E}^{-1}(h)) \\
&= \alpha_E(\trg(h))\cdot h \cdot \beta_{E^{-1}}(\src(h))
 = \alpha_E(\trg(h))\cdot h \cdot (\alpha_{E}(\src(h)))^{-1},
\end{align*} 
for any $h\in\GG(\src(E),\src(E))$.

Observe that the conjugation
$\Conj_E:\GG(\src(E),\src(E))\to\GG(\trg(E),\trg(E))$
is an isomorphism of Lie groupoids
over $\tau_E$. This implies that the derivative
$d\Conj_E$ restricts to an isomorphism
between the associated Lie algebroids.
However, the Lie algebroids associated to
the Lie groupoids $\GG(\src(E),\src(E))$ and
$\GG(\trg(E),\trg(E))$ are naturally isomorphic to the
restrictions $\fg|_{\src(E)}$ respectively 
$\fg|_{\trg(E)}$. The derivative $d\Conj_E$
therefore gives us an isomorphism of Lie algebroids
from $\fg|_{\src(E)}$ to
$\fg|_{\trg(E)}$ over $\tau_E$, which we shall
denote by
\[ \Ad_E : \fg|_{\src(E)} \to \fg|_{\trg(E)}. \]
We shall denote the induced isomorphism
of Lie algebras $\G(\fg|_{\src(E)}) \to \G(\fg|_{\trg(E)})$
again by $\Ad_E$. If $X\in\G\fg$ has
compact support in $\src(E)$, then $\Ad_E(X|_{\src(E)})$
has compact support in $\trg(E)$ and we shall write
\[ \Ad_E(X)\in\G\fg \]
for its extension by zero to all of $M$.

For any arrow $g\in E$,
the derivative of $\Lt_E$ at any arrow $h\in\GG(-,\src(g))$,
\[ (d\Lt_E)|_h : \T_h(\GG(-,\src(E)))\to \T_{gh}(\GG(-,\trg(E))), \]
and the derivative of $\Rt_E$ at any arrow $k\in\GG(\trg(g),-)$,
\[ (d\Rt_E)|_k : \T_k(\GG(\trg(E),-))\to \T_{kg}(\GG(\src(E),-)), \]
depend only on the germ of the local bisection $E$ at $g$.
This means that for any $e\in\GGG$, any
$h\in \GG(-,\src(e))$ and any $k\in\GG(\trg(e),-)$
we have isomorphisms
\[ \T_h(\GG)\to \T_{\theta(e)h}(\GG), 
\qquad \xi\mapsto d\Lt_{e}(\xi)=(d\Lt_{E})_h(\xi) \]
and
\[ \T_k(\GG)\to \T_{k\theta(e)}(\GG), 
\qquad \zeta\mapsto d\Rt_{e}(\xi)=(d\Rt_{E})_k(\zeta), \]
where $E$ is any local bisection of $\GG$ with
$\theta(e)\in E$ and $\germ_{\theta(e)}(E)=e$.

Furthermore, for any $e\in\GGG$ and any
$\xi\in\fg$ with $\pi(\xi)=\src(e)$, we can write
\[ \Ad_e(\xi)= d\Lt_{e}(d\Rt_{e}(\xi))
    =d\Rt_{e}(d\Lt_{e}(\xi)) = e\cdot \xi.\]
This defines an action of $\GGG$ on the vector bundle
$\fg$ over $M$, which we call the {\em adjoint action}.
This is indeed an action of $\GGG$ on the Lie algebroid $\fg$,
because for any local bisection $E$ of $\GG$, any $g\in E$
and any $\xi\in\fg$ with $\pi(\xi)=\src(g)$ we have
\[ \Ad_{\germ_g(E)}(\xi) = \Ad_E(\xi) .\]

Now that we have the adjoint action of $\GGG$ on $\fg$,
we obtain the convolution
$\Cc(M)/\RR$-bialgebra
\[ \Cc(\GGG,\fg) \]
of the Lie groupoid $\GG$,
as constructed in
Subsection~\ref{Bialgebra of an action}.
Recall that $\Cc(\GGG,\fg)$
is generated, as
an abelian group, by the elements of the form
\[ \langle u,\EEE \rangle, \]
where $E$ is a local bisection of $\GG$
and $u\in\Univ(\fg)$
has compact support in $\trg(E)$.

\section{Representation of the convolution bialgebra}\label{Rep}

From now on, we will assume that
$\GG$ is a Hausdorff paracompact Lie groupoid over $M$,
and we will write $\fg$ for the Lie algebroid 
$\T_M^{\trg}(\GG)$ associated to $\GG$. Recall that
we have the adjoint action of $\GGG$ on $\fg$.
In this section, we will represent the convolution
bialgebra of $\GG$ in the algebra of
$\trg$-transversal distributions on $\GG$.

\subsection{Left invariant differential operators}
Let $U$ be an open subset of $M$, and
consider the algebra $\End(\C(\GG(U,-)))$
of linear endomorphisms
of the vector space $\C(\GG(U,-))$.
For any $f\in\C(U)$,
the multiplication with the smooth function
$f\com\src|_{\GG(U,-)}$
gives us an element $\bar{f}\in\End(\C(\GG(U,-)))$.
Furthermore, for any $X\in\Gamma(\fg|_U)$,
the corresponding
left invariant vector field
$\bar{X}$ on $\GG(U,-)$, given by
$\bar{X}|_g=d\Lt_{g}(X|_{\src(g)})$ for all $g\in\GG(U,-)$,
is a partial differential 
operator on $\GG(U,-)$ of order $1$
and is therefore an element of $\End(\C(\GG(U,-)))$.
By the universal property of $\Univ(\fg|_U)$,
the maps
$\C(U)\to\End(\C(\GG(U,-)))$, $f\to\bar{f}$, and
$\Gamma(\fg|_U)\to\End(\C(\GG(U,-)))$, $X\mapsto\bar{X}$,
extend to a homomorphism of algebras 
\[ \Op=\Op_U:\Univ(\fg|_U)\to\End(\C(\GG(U,-))). \]
Endomorphisms in the image
\[ \Diff_U(\GG)=\Op_U(\Univ(\fg|_U)) \]
are linear partial
differential operators on $\GG(U,-)$ which are all
tangential to the
fibers of the target map $\trg$ and are left invariant.
In particular, this means that any operator
$D\in \Diff_U(\GG)$ can be described
as a smooth family of differential operators
$(D|_{x}:\C(\GG(U,x))\to\C(\GG(U,x)))_{x\in M}$
and the equality
\[ D|_{\src(g)}\com (\Lt_{g}^U)^{\ast}
   =(\Lt_{g}^U)^{\ast}\com D|_{\trg(g)} \]
holds for every $g\in\GG(U,-)$, where 
$(\Lt_{g}^U)^{\ast}:\C(\GG(U,\trg(g)))\to \C(\GG(U,\src(g)))$ is
the isomorphism induced by the restriction
of the left translation
diffeomorphism $\Lt_{g}$ to $\GG(U,\src(g))$.

Since differential operators are local,
the evaluation map
\[\Diff_U(\GG)\times\C(\GG(U,-))\to\C(\GG(U,-))\]
induces an action on the algebra of germs,
\[
\Diff_U(\GG)\times\C(\GG)_{g}\to\C(\GG)_{g},
\qquad (D,F_g)\mapsto D(F_g)=D(F)_g,
\]
for any $g\in\GG(U,-)$.
Furthermore, for any
open subset $W\subset\GG(U,-)$
we also have the induced action on $\C(W)$,
\[
\Diff_U(\GG)\times\C(W) \to\C(W),
\qquad (D,F)\mapsto D(F).
\]
Note that for any operator $D\in\Diff_U(\GG)$,
any local bisection $E$ of $\GG$
and any smooth function
$F\in\C(\GG(U,\trg(E)))$ we have the equality
\[ D (F \com \Lt_E|_{\GG(U,\src(E))})
   = D(F) \com \Lt_E|_{\GG(U,\src(E))}
   \in \C(\GG(U,\src(E))). \]

In~\cite{NiWeXu99}, the authors prove that
$\Op_M$ is a monomorphism and that its image 
\[ \Diff(\GG)=\Diff_M(\GG) \]
is precisely the algebra
of all linear partial differential operators on
the manifold $\GG$ which are tangential to the
fibers of the target map $\trg$ and are left invariant.
In particular, if we apply this result to
the Lie groupoid $\GG(U,U)$, it follows that the homomorphism
\[ \Univ(\fg|_U) \to \Diff(\GG(U,U)) \]
is an isomorphism.
Since any differential operator $D\in \Diff_{U}(\GG)$
is local and therefore acts also on $\GG(U,U)$,
we have the natural homomorphism of algebras
\[ \Diff_{U}(\GG) \to \Diff(\GG(U,U)). \]
By the left invariance, this homomorphism is in fact
an isomorphism of algebras. This yields that the homomorphism
\[ \Op_U:\Univ(\fg|_U)\to\Diff_U(\GG) \]
is also an isomorphism of algebras.

To simplify the notation, we will identify
the algebra $\Cc(U)$  with the subalgebra
of $\C(M)$, which consists of the smooth functions
on $M$ with compact support in $U$.
The elements of $\Univ(\fg)$ with compact support
in $U$ form the subalgebra
\[ \Univ_{c,U}(\fg) = \Cc(U) \cdot \Univ(\fg) \]
of $\Univ(\fg)$. The corresponding differential operators
on $\GG$ form the subalgebra
\[ \Diff_{c,U}(\GG)=
   \Op(\Univ_{c,U}(\fg))
   = \Cc(U)\cdot \Diff(\GG) \]
of $\Diff(\GG)$.
Note that we can identify
\[\Univ_{c,U}(\fg)=
  \Univ_{c,U}(\fg|_U)
  = \Cc(U)\cdot \Univ(\fg|_U)\]
and
\[ \Diff_{c,U}(\GG)=
\Diff_{c,U}(\GG(U,U))=\Cc(U)\cdot \Diff_{U}(\GG).
\]

The algebra isomorphism
$\Op:\Univ(\fg)\to\Diff(\GG)$
is also an isomorphism of left $\C(M)$-modules and
it induces the algebra isomorphism of localizations
\[ \Op_x:\Univ(\fg)_x\to\Diff(\GG)_x, \]
for any $x\in M$.

\begin{prop}\label{Action of U on smooth functions on G}
For any $x\in M$ and any $g\in\GG(x,-)$, the
evaluation map $\Diff(\GG)\times\C(\GG)\to\C(\GG)$
induces an action 
\[
\Univ(\fg)_{x}\times\C(\GG)_{g}\to\C(\GG)_{g}.
\]
We shall denote this action as
\[ (u_x,F_g)\mapsto u_x[F_g]=(\Op(u)(F))_g, \]
for any $u\in \Univ(\fg)$ and
any $F\in\C(\GG)$.
\end{prop}

\begin{proof}
We already noted that we have the induced action of
$\Diff(\GG)$ on $\C(\GG)_{g}$.
It is now enough to show that for any
$D\in\N_x(\Diff(\GG))$ and any $F\in\C(\GG)$
we have $D(F)_g=0$. To this end, choose a function
$f\in\C(M)$ such that $\bar{f} D=D$ and $f_x=0$.
It follows that $(f\com \src)_g=0$ and
\[
D(F)_g=((\bar{f}D)(F))_g=(f\com\src)_g D(F)_g=0.\qedhere
\]
\end{proof}

Let $E$ be a local bisection of $\GG$. Recall that
the conjugation isomorphism of Lie groupoids
$\Conj_E:\GG(\src(E),\src(E))\to\GG(\trg(E),\trg(E))$
induces the isomorphism of Lie algebroids
$\Ad_E:\fg|_{\src(E)}\to \fg|_{\trg(E)}$
(see Subsection~\ref{The adjoint action}),
which in turn induces the isomorphism of bialgebras
\[ \Univ(\Ad_E):
   \Univ(\fg|_{\src(E)})
   \to
   \Univ(\fg|_{\trg(E)}) \]
over $\C(\tau_E^{-1})$ (see
Subsection~\ref{Actions of etale Lie groupoids on Lie algebroids}).
On the other hand, the isomorphism
$\Conj_E$ also induces an isomorphism of
algebras
\[ \Diff(\Conj_E):
   \Diff(\GG(\src(E),\src(E)))
   \to
   \Diff(\GG(\trg(E),\trg(E))) \]
such that 
the diagram of isomorphisms
\[
\xymatrix@C+18pt{
\Univ(\fg|_{\src(E)})
   \ar[r]^{\Univ(\Ad_E)} \ar[d]_{\Op}
&
\Univ(\fg|_{\trg(E)})
   \ar[d]^{\Op} \\
\Diff(\GG(\src(E),\src(E)))
   \ar[r]^{\Diff(\Conj_E)}
&
\Diff(\GG(\trg(E),\trg(E)))
}
\]
commutes. Explicitly, for any 
$D\in\Diff(\GG(\src(E),\src(E)))$ and any
function
$F\in\C(\GG(\trg(E),\trg(E)))$ we have
\[ \Diff(\Conj_E)(D)(F)\com\Conj_{E}
   =D(F\com\Conj_E).
\]

The isomorphism $\Diff(\Conj_E)$
can be in fact extended, by left invariance,
to the algebra $\Diff_{\src(E)}(\GG)$.
Define a map
\[ \overline{\Ad}_{E}:
\End(\C(\GG(\src(E),-))) \to \End(\C(\GG(\trg(E),-)))
\]
by
\[
\overline{\Ad}_E(D)(F)
=
D(F\com \Rt_E^{-1}) \com \Rt_E,
\]
for any $D\in\End(\C(\GG(\src(E),-)))$
and any $F\in\C(\GG(\trg(E),-))$.
It is clear that
$\overline{\Ad}_E$ is an isomorphism of algebras
with inverse $\overline{\Ad}_{E^{-1}}$.
Furthermore:

\begin{prop}
For any local bisection $E$ of $\GG$,
the diagram
\[
\xymatrix{
\Univ(\fg|_{\src(E)})
   \ar[r]^{\Univ(\Ad_E)} \ar[d]_{\Op}
&
\Univ(\fg|_{\trg(E)})
   \ar[d]^{\Op} \\
\End(\C(\GG(\src(E),-)))
   \ar[r]^{\overline{\Ad}_E}
&
\End(\C(\GG(\trg(E),-)))
}
\]
commutes. In particular we have
$\overline{\Ad}_E(\Diff_{\src(E)}(\GG))=\Diff_{\trg(E)}(\GG)$.
\end{prop}

\begin{proof}
It is enough to check the equality
$\overline{\Ad}_E\com\Op =\Op\com\Univ(\Ad_E)$
on all
$f\in\C(\src(E))$ and $X\in\G(\fg|_{\src(E)})$.

For any smooth function $F\in\C(\GG(\trg(E),-))$
and any $g\in\GG(\trg(E),-)$
we have
\begin{align*}
\overline{\Ad}_E(\Omega(f))(F)(g)
&=\overline{\Ad}_E(\bar{f})(F)(g)
=(\bar{f}(F\com \Rt_E^{-1}) \com \Rt_E )(g) \\
&=(((f\com\src)(F\com\Rt_E^{-1})) \com \Rt_E )(g) \\
&=(f\com\src\com\Rt_E)(g)(F\com\Rt_E^{-1} \com \Rt_E)(g) \\
&=f(\src(g \beta_E(\src(g)))) F(g)
=f(\src(\beta_E(\src(g)))) F(g) \\
&=(f\com\tau_E^{-1})(\src(g)) F(g)
= (((f\com\tau_E^{-1})\com\src) F )(g) \\
&= \Omega(f\com\tau_E^{-1})(F)(g)
= \Omega(\Univ(\Ad_E)(f))(F)(g).
\end{align*}
Furthermore, we compute
\begin{align*}
\overline{\Ad}_E(\Op(X))(F)(g)
&= \overline{\Ad}_E(\bar{X})(F)(g)
=(\bar{X}(F\com \Rt_E^{-1})\com\Rt_E)(g) \\
&=(\Rt_E^{-1})_\ast(\bar{X}) (F)(g),
\end{align*} 
so $\overline{\Ad}_E(\Op(X))$ 
is the vector field
$(\Rt_E^{-1})_\ast(\bar{X})$
corresponding to $\bar{X}$
along the diffeomorphism $\Rt_E^{-1}$.
The vector field $\overline{\Ad}_E(\Op(X))$
is tangent to the $\trg$-fibers
and is left invariant, because
the right translations preserve
the $\trg$-fibers and commute with
the left translations.

For any $x\in\src(E)$, 
the value of the vector field
$(\Rt_E^{-1})_\ast(\bar{X})$ at $1_{\tau_E(x)}$
equals to
\begin{align*}
(\Rt_E^{-1})_\ast(\bar{X})|_{1_{\tau_E(x)}}
&= d(\Rt_E^{-1})(\bar{X}|_{\Rt_E(1_{\tau_E(x)})}) 
= d(\Rt_E^{-1})(\bar{X}|_{\alpha_E(x)}) \\
&= d(\Rt_E^{-1})(d(\Lt_E)(\bar{X}|_{1_x}))
= d(\Conj_E)(X|_{x}) \\
&= (\Ad_E(X))|_{\tau_E(x)}.
\end{align*} 
This means that the left invariant vector fields
$\overline{\Ad}_E(\Omega(X))$ and $\Omega(\Ad_E(X))$
agree on $\trg(E)\subset \GG$, so they must be equal
on $\GG(\trg(E),-)$.
\end{proof}

As a consequence,
the isomorphism $\overline{\Ad}_E$ also gives us
an isomorphism
\[
\overline{\Ad}_E:
\Diff_{c,\src(E)}(\GG)
\to
\Diff_{c,\trg(E)}(\GG)
\]
such that the diagram of isomorphisms
\[
\xymatrix{
\Univ_{c,\src(E)}(\fg)
   \ar[r]^{\Univ(\Ad_E)} \ar[d]_{\Op}
&
\Univ_{c,\trg(E)}(\fg)
   \ar[d]^{\Op} \\
\Diff_{c,\src(E)}(\GG)
   \ar[r]^{\overline{\Ad}_E}
&
\Diff_{c,\trg(E)}(\GG)
}
\]
commutes.

\subsection{Transversal distributions}
Let $N$ be a paracompact Hausdorff smooth manifold.
Recall that a distribution on $N$ is a
linear functional on $\C(N)$ which is continuous
with respect to the Fr\'{e}chet topology.
We will write $\cE'(N)$ for the vector space of
distributions on $N$.

Let $p:N\to M$ be a smooth surjective submersion.
We have the 
induced action of $\Cc(M)$ on $\C(N)$, defined by
$f\cdot F=(f\com p)F$,
for $f\in\Cc(M)$ and $F\in\C(N)$.
A $p$-transversal distribution on $N$ is then
defined to be a $\Cc(M)$-linear continuous map
$\C(N)\to\Cc(M)$, where 
we equip the spaces $\C(N)$ and $\Cc(M)$
with the Fr\'{e}chet topology respectively the $LF$-topology.
We will denote the vector space of $p$-transversal 
distributions by
\[ \cE'_{p}(N).\]

Any $p$-transversal distribution
$T\in\cE'_{p}(N)$ defines
a family $(T|_{x})_{x\in M}$ of distributions
$T|_{x}\in\cE'(p^{-1}(x))$,
which is smooth in the sense that the function
$x\mapsto T|_{x}(F|_{p^{-1}(x)})$ on $M$
is smooth
for every $F\in\C(N)$.
One can define the support of a transversal distribution 
similarly as in the ordinary case.

\begin{ex} \rm
(1)
Let $Q$ be a paracompact Hausdorff
manifold and let
$\pr:M\times Q\to M$ be the projection.
Then we can identify the space $\cE'_{\pr}(M\times Q)$
with the space $\Cc(M,\cE'(Q))$
of smooth, compactly supported $\cE'(Q)$-valued
functions on $M$. Since $\cE'(Q)$ is a complete locally
convex space, a function $T:M\to\cE'(Q)$ is smooth
if and only if the function $\varphi\com T:M\to\RR$
is smooth for every continuous linear functional
$\varphi:\cE'(Q)\to\RR$ (see~\cite{KrMi97} for details). 

(2)
For any local bisection $E$ of $\GG$ and any function
$f\in\C(M)$ with compact support in $\src(E)$ we define
a $\trg$-transversal distribution
$\llbracket E,f\rrbracket \in\cE'_{\trg}(\GG)$ so that
for any $F\in\C(\GG)$, the function
$\llbracket E,f\rrbracket(F)$ has compact support in $\trg(E)$
and
\[
\llbracket E,f\rrbracket(F)(x)=f(\tau_E^{-1}(x))F(\beta_{E}(x)),
\qquad  x\in \trg(E).
\]
The distribution $\llbracket E,f\rrbracket$ can be seen
as a  smooth family of Dirac measures defined at the points of $E$
and weighted by the function $f\com\tau_E^{-1}$.
Explicitly, we have
\[
\llbracket E,f\rrbracket|_{x}=f(\tau_E^{-1}(x))\delta_{\beta_{E}(x)},
   \qquad x\in \trg(E).
\]
The support of $\llbracket E,f\rrbracket$ is contained in $E$.

(3) 
Let $E$ be a local bisection of $\GG$ 
and let $D\in\Diff_{c,\src(E)}(\GG)$.
We define the distribution
$\llbracket E,D\rrbracket\in\cE'_{\trg}(\GG)$
so that for any $F\in\C(\GG)$,
the function
$\llbracket E,D\rrbracket(F)$ has compact support
in $\trg(E)$ and
\[
\llbracket E,D\rrbracket(F)(x)=D(F)(\beta_{E}(x)),
\qquad  x\in \trg(E).
\]
This transversal distribution computes
$\trg$-vertical derivatives 
of a function $F$
along the local section $E$.
We can describe this transversal distribution
as a smooth family of distributions along
the fibers of $\trg$ as follows:
For any $x\in M$ and $g\in \GG(-,x)$ we denote by
$\text{ev}_{g}:\C(\GG)\to\RR$ the evaluation map
at the point $g$. Since $D$ is tangential to the fibers
of $\trg$,
it restricts to an endomorphism of $\C(\GG(-,x))$,
and the composition
$D|_{g}=\text{ev}_{g}\com D:\C(\GG(-,x))\to\RR$
belongs to $\cE'(\GG(-,x))$. Using this notation,
we can now express
\[
\llbracket E,D\rrbracket |_{x}=D|_{\beta_{E}(x)},
\qquad x\in \trg(E).
\]
\end{ex}

It turns out that there is a natural product
of $\trg$-transversal distributions which makes
\[ \cE'_{\trg}(\GG) \]
into an algebra
(see~\cite{LeMaVa17}).
In fact, in~\cite{LeMaVa17} (see also~\cite{ErYu19}),
the authors constructed three versions
of distributional algebras on $\GG$:
the algebras $\cE'_{\trg}(\GG)$,
$\cE'_{\src}(\GG)$ and $\cE'_{\trg,\src}(\GG)$
of $\trg$-transversal, $\src$-transversal
and bitransversal distributions, respectively.
For our purposes, the algebra $\cE'_{\trg}(\GG)$
will be most suitable. 

Let us recall the definition of the product
$T'\ast T$ of distributions $T,T'\in\cE'_{\trg}(\GG)$:
For any function $F\in\C(\GG)$ and any arrow
$g\in\GG$ we have the function
$F\com \Lt_{g}\in\C(\GG(-,\src(g)))$,
and it turns out that the function
$\GG\to\RR$, $g\mapsto T|_{\src(g)}(F\com \Lt_{g})$,
is smooth. The distribution
$T'\ast T\in \cE'_{\trg}(\GG)$ is given by
\[
(T'\ast T)(F)(x)=
  T'\left(g\mapsto T|_{\src(g)}(F\com \Lt_{g})\right)(x),
\]
for any $F\in\C(\GG)$ and any $x\in M$.

\begin{prop}\label{Prop-product-of-basic-distributions}
For any local bisections $E$ and $E'$ of the Lie groupoid $\GG$,
any $D\in\Diff_{c,\src(E)}(\GG)$ and any
$D'\in\Diff_{c,\src(E')}(\GG)$ we have
\[  \llbracket E',D'\rrbracket \ast \llbracket E,D\rrbracket
= \llbracket E'\cdot E,\overline{\Ad}_{E^{-1}}(D') D\rrbracket. \]
\end{prop}

\begin{proof}
Let us abbreviate $T=\llbracket E,D\rrbracket$
and $T'=\llbracket E',D'\rrbracket$. 
For any function $F\in\C(\GG)$ and any $x\in \trg(E'\cdot E)$ we have
\begin{align*}
(T'\ast T)(F)(x)
&=T' \left(g\mapsto T|_{\src(g)}(F\com \Lt_{g})\right)(x) \\
&=T' \left(g\mapsto D(F\com L_{g})(\beta_E(\src(g))) \right)(x) \\
&=T' \left(g\mapsto (D(F) \com L_{g})
     (\beta_E(\src(g))) \right)(x) \\
&=T' \left(g\mapsto (D(F)\com\Rt_E\com\Rt_E^{-1}\com L_{g})
     (\beta_E(\src(g))) \right)(x) \\
&=T' \left(g\mapsto (D(F)\com\Rt_E)(g) \right)(x) \\
&=D'(D(F)\com\Rt_E)(\beta_{E'}(x)) \\
&=\left(\overline{\Ad}_{E^{-1}}(D')(D(F))\com\Rt_E\right)
  (\beta_{E'}(x)) \\
&=\overline{\Ad}_{E^{-1}}(D')(D(F))(\beta_{E'\cdot E}(x)) \\
&=\llbracket E'\cdot E,\overline{\Ad}_{E^{-1}}(D') D\rrbracket (x)
\end{align*}
because $D$ is left invariant
and  $\Rt_E^{-1}(L_g(\beta_E(\src(g))))=g$.
\end{proof}

\subsection{Representation of the convolution bialgebra by distributions}
In the rest of this section we will construct
a representation of the algebra
$\Cc(\GGG,\fg)$ in the algebra of distributions
$\cE'_{\trg}(\GG)$.

We define a map
\[ \Phi=\Phi_{\GG}:\Cc(\GGG,\fg)\to\cE'_{t}(\GG) \]
by the formula
\[
\Phi(a)(F)(x)=
\!\!\!\sum_{e\in\GGG(-,x)}\!\!\!
(e^{-1}\cdot a(e))[F_{\theta(e)}](\theta(e)),
\]
for any $a\in\Cc(\GGG,\fg)$, any $F\in\C(\GG)$ and any $x\in M$.
It is clear that $\Phi(a)(F)(x)$
is well-defined and linear in $a$.
The fact that $\Phi$ actually maps $\Cc(\GGG,\fg)$
into $\cE'_{\trg}(\GG)$ follows from the next example:

\begin{ex}\label{Ex-Phi-on-basic-elements} \rm 
Let $E$ be a local bisection of $\GG$ and suppose that
$u\in\Univ(\fg)$ has compact support in $\trg(E)$.
Then we have
\begin{align*}
\Phi(\langle u,\EEE \rangle)(F)(x)
&=( (\germ_{\beta_E(x)}(E))^{-1}\cdot u_{x} )
  [F_{\beta_E(x)}](\beta_E(x)) \\
&=( \germ_{\alpha_{E^{-1}}(x)}(E^{-1})\cdot u_{x} )
  [F_{\beta_E(x)}](\beta_E(x)) \\
&=( \Ad_{E^{-1}}(u) )_{\tau^{-1}_E(x)}
  [F_{\beta_E(x)}](\beta_E(x)) \\
&=\Op(\Ad_{E^{-1}}(u))(F)(\beta_E(x)) \\
&=\overline{\Ad}_{E^{-1}}(\Op(u))(F)(\beta_E(x))
\end{align*}
for any $F\in\C(\GG)$ and any $x\in\trg(E)$, which shows that
\[
\Phi(\langle u,\EEE\rangle)=
\llbracket E,\overline{\Ad}_{E^{-1}}(\Op(u))\rrbracket
\in \cE'_{\trg}(\GG).
\]
Since $\Phi$ is additive and every element of
$\Cc(\GGG,\fg)$ can be written as a finite sum
of elements of this form, we conclude that $\Phi$
really maps into the space $\cE'_{\trg}(\GG)$. 

In particular, for any function $f\in\C(M)$ with compact support
in $\trg(E)$ we have
\[
\Phi(\langle f,\EEE\rangle)
=\llbracket E,\overline{\Ad}_{E^{-1}}(\bar{f})\rrbracket
=\llbracket E,\Omega(\Ad_{E^{-1}}(f))\rrbracket
=\llbracket E,\overline{f\com\tau_E}\rrbracket
=\llbracket E,f\com\tau_E\rrbracket.
\]
If $u\in\Univ(\fg)$ has
compact support in $M$, then
\[
\Phi(\langle u,M\rangle)=\llbracket M,\Op(u)\rrbracket.
\]
\end{ex}

Finally, we will show that
$\Phi$ is a homomorphism of algebras
and we will explicitly describe its kernel:

\begin{theo}
Let $\GG$ be a Hausdorff paracompact Lie groupoid
with the associated Lie algebroid $\fg$. Then we have:

(i) The map 
\[ \Phi_{\GG} : \Cc(\GGG,\fg) \to \cE'_{t}(\GG) \]
is a homomorphism of algebras.

(ii)
An element $a\in \Cc(\GGG,\fg)$ is in the kernel
of the homomorphism $\Phi_{\GG}$, if, and only if, we have
\[ \sum_{e\in\theta^{-1}(g)} \!\!\! e^{-1}\cdot a(e) = 0  \]
for all $g\in \GG$.
\end{theo}

\begin{proof}
(i)
Choose local bisections $E$ and $E'$ of $\GG$
and let $u,u'\in\Univ(\fg)$ 
have compact supports contained in $\trg(E)$ respectively $\trg(E')$.
By Proposition~\ref{Prop-product-of-basic-distributions}
and Example~\ref{Ex-Phi-on-basic-elements} we get
\begin{align*}
\Phi(\langle u',\EEEp\rangle)
& \ast\Phi(\langle u,\EEE\rangle) =
\llbracket E',\overline{\Ad}_{E'^{-1}}(\Op(u'))\rrbracket
\ast
\llbracket E,\overline{\Ad}_{E^{-1}}(\Op(u))\rrbracket \\
&=
\llbracket
E'\cdot E,
\overline{\Ad}_{E^{-1}}
   \left( \overline{\Ad}_{E'^{-1}}(\Op(u'))\right) 
   \overline{\Ad}_{E^{-1}}(\Op(u) )
\rrbracket \\
&=
\llbracket
E'\cdot E,
\overline{\Ad}_{E^{-1}}
     \left(
     \overline{\Ad}_{E'^{-1}}(\Op(u')) \Op(u)
     \right)
\rrbracket \\
&=
\llbracket
E'\cdot E,
\overline{\Ad}_{E^{-1}}
     \left(
     \overline{\Ad}_{E'^{-1}}
         \left(
         \Op(u') \overline{\Ad}_{E'}(\Op(u))
         \right)
     \right)
\rrbracket \\
&=
\llbracket
E'\cdot E,
\overline{\Ad}_{(E'\cdot E)^{-1}}
         \left(
         \Op(u') \Op(\Ad_{E'}(u))
         \right)
\rrbracket \\
&=
\llbracket
E'\cdot E,
\overline{\Ad}_{(E'\cdot E)^{-1}}
         \left(
         \Op(u'\Ad_{E'}(u))
         \right)
\rrbracket \\
&=
\Phi(
\langle
u'\Ad_{E'}(u)
,\EEEp\cdot\EEE
\rangle)
=
\Phi(
\langle u',\EEEp\rangle
\cdot
\langle
u,\EEE
\rangle
).
\end{align*}
This shows that $\Phi$ preserves product
on elements which are supported on local bisections.
Since the set of such elements generates the abelian group
$\Cc(\GGG,\fg)$, this is sufficient.

(ii,$\Leftarrow$)
Assume that $a\in \Cc(\GGG,\fg)$ satisfies
$\sum_{e\in\theta^{-1}(g)} e^{-1}\cdot a(e) = 0$
for all $g\in \GG$.
Then we have
\begin{align*}
\Phi(a)(F)(x)
&=
\!\!\!\sum_{e\in\GGG(-,x)}\!\!\!
(e^{-1}\cdot a(e))[F_{\theta(e)}](\theta(e)) \\
&=
\sum_{g\in\GG(-,x)} \Big(
\sum_{e\in\theta^{-1}(g)}\!\!\!
e^{-1}\cdot a(e) \Big)
[F_{g}](g) = 0
\end{align*}
for any $x\in M$ and any $F\in\C(\GG)$,
thus $\Phi(a)=0$.

(ii,$\Rightarrow$)
Take any $a\in \Cc(\GGG,\fg)$ with $\Phi(a)=0$.
For any $x\in M$
and any function $F\in\C(\GG)$
we have
\begin{align*}
0=\Phi(a)(F)(x)
&=
\sum_{g\in\GG(-,x)} \Big(
\sum_{e\in\theta^{-1}(g)}\!\!\!
e^{-1}\cdot a(e) \Big) [F_{g}](g).
\end{align*}
For any $g\in\GG(-,x)$ we can find a new function
$F'\in\C(\GG)$ such that $F'_g=F_g$ and
$F'_{\theta(e)}=0$ for any $e\in\GGG(-,x)$ with
$\theta(e)\neq g$ and $a(e)\neq 0$.
This implies that we have in fact
\[ \sum_{e\in\theta^{-1}(g)}\!\!\!
   (e^{-1}\cdot a(e)) [F_{g}](g) = 0 \]
for any $g\in \GG$ and any $F\in\C(\GG)$.

Write $a=\sum_{i=1}^m \langle u_i,\EEE_i\rangle$
for some local bisections $E_i$ of $\GG$ and
elements $u_i\in\Univ(\fg)$ with compact support in $\trg(E_i)$,
$i=1,\ldots,m$.
For any $g\in\GG$, let $\nu(g)$ be the number of elements
of the set 
$ \{ \germ_g(E_i)\in\GGG \,;\, i=1,\ldots,m,\, g\in E_i \}$
and put $A(g)=\{ i\in\{1,\ldots,m\} \,;\, g\in E_i \}$.

For any $g\in\GG$ with $\nu(g)=0$ we clearly have
$a(e)=0$ for any $e\in\theta^{-1}(g)$.

Take any $g\in\GG$ with $\nu(g)=1$.
We can find a small open neighbourhood $W$ of $g$ in $\GG$
which does not intersect the compact sets
$\beta_{E_j}(\supp(u_j))$, $j\in\{1,\ldots,m\}\setminus A(g)$.
We can choose a local bisection $E$ of $\GG$ such that
$g\in E\subset W$ and $E\subset E_i$ for any $i\in A(g)$,
because $\nu(g)=1$.
For any $h\in E$ and any $F\in\C(\GG)$ we then have
\begin{align*}
0 &=\sum_{e\in\theta^{-1}(h)}\!\!\!
   (e^{-1}\cdot a(e)) [F_{h}](h)
=  \big(\germ_h(E)^{-1}\cdot a(\germ_h(E))\big) [F_{h}](h) \\
&=\sum_{i\in A(g)}
\big(\germ_{h^{-1}}(E^{-1})\cdot (u_i)_{\trg(h)}\big) [F_{h}](h) 
=\sum_{i\in A(g)}
(\Ad_{E^{-1}}(u_i))_{\src(h)} [F_{h}](h) \\
&=\sum_{i\in A(g)}
\big( \Op(\Ad_{E^{-1}}(u_i))(F)\big)_{h} (h)
=\sum_{i\in A(g)}
\Op(\Ad_{E^{-1}}(u_i))(F)(h) \\
&=\sum_{i\in A(g)}
\Op(\Ad_{E^{-1}}(u_i))(F)(\Lt_E(1_{\src(h)}))
=
\sum_{i\in A(g)}
\Op(\Ad_{E^{-1}}(u_i))(F\com \Lt_E)(1_{s(h)}).
\end{align*}
This shows that
$\sum_{i\in A(g)} (\Ad_{E^{-1}}(u_i))=0$ on
$\src(E)$, and in particular
\begin{align*}
(\germ_{g}(E))^{-1}\cdot & \,a(\germ_{g}(E))
= (\germ_{g}(E))^{-1}\cdot \sum_{i\in A(g)}
  (u_i)_{\trg(g)} \\
&= \sum_{i\in A(g)} (\Ad_{E^{-1}}(u_i))_{\src(g)} 
= 0 ,
\end{align*}
so $a(\germ_{g}(E))=0$. 
We can conclude that $a(e) = 0$
for any $g\in\GG$ with $\nu(g)=1$
and any $e\in\theta^{-1}(g)$.

Finally, choose any $g\in\GG$ and put $y=\src(g)$.
Let $B$ be the set of indexes
$j\in \{1,\ldots,m\}$ with $y\in\src(E_j)$.
Choose a small open neighbourhood
$U$ of $y$ in $M$ such that $U\subset\src(E_j)$
for all $j\in B$.
We can take $U$ so small
that $\alpha_{E_i}(U)$ does not intersect $\alpha_{E_j}(U)$,
for all $i\in A(g)$ and $j\in B\setminus A(g)$. Next,
for any $j\in \{1,\ldots,m\}\setminus B$ we can
represent the element $\langle u_j,\EEE_j\rangle$
by a smaller local bisection and shrink $U$ further
so that $U$ does not intersect $\src(E_j)$.
Note that this modification does not increase the values
of the function $\nu$.

For any $z\in U$ and $i\in A(g)$ write
$e_{i,z}=\germ_{\alpha_{E_i}(z)}(E_i)$, so
the germ of the element
\[
u=\sum_{i\in A(g)} \Ad_{E_i^{-1}}(u_i)|_{U} \in\Univ(\fg|_U)
\]
at $z$ equals
\[ u_z = 
\sum_{i\in A(g)} e_{i,z}^{-1}\cdot 
     (u_i)_{\tau_{E_i}(z)} \in\Univ(\fg)_{z}.\]

For any $z\in U$ denote by
$\mu(z)$ the number of elements of the set
\[ I(z)=\{ \alpha_{E_i}(z) \,;\, i\in A(g) \}. \]
Observe that the function $\mu$ is lower semi-continuous.
In particular, if $V$ is any non-empty
open subset of $U$
and $m_V$ is the maximal value
of $\mu|_V$, then the non-empty
set $O_V=(\mu|_V)^{-1}(m_V)$ is open in $V$.
Now observe that for any $z\in O_V$ and
any $h\in I(z)$ we have
$\nu(h)= 1$. In other words,
there is exactly one element, say $e(h)$,
in the intersection
$\theta^{-1}(h)\cap \{ e_{i,z} \,;\, i\in A(g) \}$.
By the argument above it follows that
\[ a(e(h))=0 \]
for any $z\in O_V$ and
any $h\in I(z)$.
This implies
\begin{align*}
 u_z 
&= 
\sum_{i\in A(g)} e_{i,z}^{-1}\cdot (u_i)_{\tau_{E_i}(z)}
= \sum_{h\in I(z)} \,\,
  \sum_{{\substack{i\in A(g); \\ h\in E_i}}}
   e(h)^{-1}\cdot (u_i)_{\tau_{e(h)}(z)} \\
&= \sum_{h\in I(z)}
  e(h)^{-1}\cdot \sum_{\substack{i\in A(g); \\ h\in E_i}} 
   (u_i)_{\tau_{e(h)}(z)} 
=\sum_{h\in I(z)}
  e(h)^{-1}\cdot a(e(h))=0  
\end{align*}
for any $z\in O_V$.
Since this is true for any non-empty open subset $V$ of $U$,
the set of points $z\in U$ for which
$u_z=0$ is dense in $U$.
By continuity it follows that $u=0$.
In particular, we have
\begin{align*}
\sum_{e\in\theta^{-1}(g)} e^{-1}\cdot a(e)
&=
\sum_{e\in\theta^{-1}(g)} e^{-1}\cdot
   \sum_{\substack{i\in  A(g); \\  e_{i,y}=e}}
   (u_i)_{\trg(e)}
=
\sum_{e\in\theta^{-1}(g)} \,
   \sum_{\substack{i\in  A(g); \\  e_{i,y}=e}}
   e_{i,y}^{-1}\cdot (u_i)_{\trg(e_{i,y})}  \\
&=
\sum_{i\in A(g)} e_{i,y}^{-1}\cdot (u_i)_{\tau_{E_i}(y)}
= u_y =0. \qedhere
\end{align*}
\end{proof}

\begin{ex} \rm
(1) 
If $E$ is a local bisection of $\GG$,
$u\in\Univ(\fg)$ has compact support in $\trg(E)$ and
$\Phi_{\GG}(\langle u,\EEE\rangle)=0$, then we have
$u=0$. It follows that
$\Phi_{\GG}$ is a monomorphism on the image of the
algebra embedding
$\Univ_{c,M}(\fg)\to\Cc(\GGG,\fg)$, $u\mapsto\langle u,M\rangle$.

(2)
Choose an increasing function $\varphi\in\C(\RR)$
such that $\frac{d^n\varphi}{dt^{n}}(0)=0$ for any
$n\in\{0,1,\ldots\}$. For any $i,j\in\{0,1\}$
let $f_{i,j}:\RR\to\RR$ be the diffeomorphism given by
\[ f_{i,j}(t)=
\left\{
\begin{array}{lcl}
t+ 2^i\varphi(t)\,;& t\leq 0, \\
t+ 2^j\varphi(t)\,;& t\geq 0,
\end{array}
\right. 
\]
and let $E_{i,j}\subset \RR\times\RR$ be the graph
of this diffeomorphism.
Note that each $E_{i,j}$ is a local
bisection (and, in this case, a global bisection)
of the pair Lie groupoid $\RR\times \RR$ over $\RR$
with the associated Lie algebroid $\T(\RR)$.
Choose a function $f\in\Cc(\RR)$ with $f(0)\neq 0$ and
consider the element
\[ a=\!\!\!\sum_{i,j\in\{0,1\}}\!\!\!
   \langle (-1)^{i+j}f,\EEE_{i,j}\rangle 
   \in \Cc((\RR\times \RR)^{\scriptscriptstyle\#})
   \subset\Cc((\RR\times \RR)^{\scriptscriptstyle\#},\T(\RR)).
\]
We have $a\neq 0$ and $\Phi_{\RR\times\RR}(a)=0$,
so $\Phi_{\RR\times\RR}$ is not a monomorphism.

(3)
If $\GG$ is a Lie groupoid over a discrete space,
then $\Phi_{\GG}$ is a monomorphism.

(4) If $\EE$ is a Hausdorff paracompact
\'etale Lie groupoid over $M$, then
$\T^{\trg}_M(\EE)=0$,
$\Cc(\EE^{\scriptscriptstyle\#},\T^{\trg}_M(\EE))=\Cc(\EE)$
and $\Phi_{\EE}$ is an isomorphism.

(5)
Let $K$ be a Lie group with Lie algebra $\fk$,
and let $\Univ(\fk)$ be the universal enveloping algebra of $\fk$.
Then $K^{\scriptscriptstyle \#}$ is the set $K$ with
discrete  topology,
so any element
$a\in\Cc(K^{\scriptscriptstyle \#},\fk)$
can be written as a finite sum
\[a=\sum\limits_{i=1}^{n}\langle u_{i},k_{i}\rangle,\]
for some elements 
$k_{1},\ldots,k_{n}\in K$
and some $u_{1},\ldots,u_{n}\in\Univ(\fk)$.
Under the
above representation,
the element $\langle u_{i},k_{i}\rangle$ is mapped
to a distribution on $K$ which corresponds to
the evaluation of the 
left invariant operator
$\Op(\Ad_{k_{i}^{-1}}(u_i))$ at the point $k_{i}$. 
The image of the homomorphism
$\Phi_{K}$ is the space $\cE'_{\text{fin}}(K)$
of distributions with
finite support on $K$.

(6)
The homomorphism
$\Phi_{\GG}$ can be used to define a natural locally convex 
topology on the algebra $\Cc(\GGG,\fg)$. We have the $\Cc(M)$-bilinear
pairing 
\[
\Cc(\GGG,\fg)\times\C(\GG)\to\Cc(M),\qquad (a,F)\mapsto\Phi_{\GG}(a)(F).
\] 
For any bounded subset $B\subset\C(\GG)$ with respect
to the Fr\'echet topology
and any continuous seminorm $q$ on $\Cc(M)$
with respect to the $LF$-topology
we define a seminorm $p_{B,q}$
on $\Cc(\GGG,\fg)$ by
\[
p_{B,q}(a)=\sup\big\{\,q\big(\Phi_{\GG}(a)(F)\big)\,;\, F\in B \big\}.
\]
The topology on $\Cc(\GGG,\fg)$ is then defined by the family $\{p_{B,q}\}$ as $B$ 
ranges over all bounded subsets of $\C(\GG)$ and $q$ over all
continuous seminorms on $\Cc(M)$. In general, this topology is not Hausdorff
and we have $\ker(\Phi_{\GG})=\mathrm{Cl}_{\Cc(\GGG,\fg)}\{0\}$.
\end{ex}

\end{document}